\documentclass{amsart}
\title{Some mutually inconsistent generic large cardinals}
\author{Monroe Eskew}

\address{International Research Fellow of the Japan Society for the Promotion of Science \\
Institute of Mathematics \\
University of Tsukuba \\
1-1-1 Tennodai \\
Tsukuba, Ibaraki 305-8571 \\ Japan. }

\usepackage{color} 
\usepackage{amssymb} 
\usepackage{amsmath} 
\usepackage[latin1]{inputenc} 
\usepackage{enumerate}
\usepackage{amsthm}
\usepackage{cite}

\date{}

\newtheorem{theorem}{Theorem}[section]
\newtheorem{lemma}[theorem]{Lemma}
\newtheorem{proposition}[theorem]{Proposition}
\newtheorem{corollary}[theorem]{Corollary}

\newtheorem*{definition}{Definition}

\DeclareMathOperator{\dom}{dom}

\DeclareMathOperator{\ot}{ot}

\DeclareMathOperator{\p}{\mathcal{P}}
\DeclareMathOperator{\den}{d}
\DeclareMathOperator{\sat}{sat}
\DeclareMathOperator{\add}{Add}
\DeclareMathOperator{\col}{Col}
\DeclareMathOperator{\ord}{Ord}

\DeclareMathOperator{\reg}{Reg}
\DeclareMathOperator{\supp}{supp}
\DeclareMathOperator{\crit}{crit}

\begin{document}
\maketitle
\thispagestyle{empty}
\pagestyle{empty} 

In \cite{thesis}, a method was developed to show some non-implications between certain strong combinatorial properties of ideals.  It was shown that one may force to rid the universe of ideals of minimal density on a given successor cardinal above $\omega_1$, while preserving ``nonregular'' ideals.  Taylor proved that these are actually equivalent properties for ideals on $\omega_1$ \cite{t1}.  The author realized that the method was also able to show that a broad class of ``generic large cardinals'' are individually consistent (relative to conventional large cardinals), yet mutually contradictory.

In \cite{potentaxioms} and \cite{foremanhandbook}, Foreman proposed generic large cardinals as new axioms for set theory.  The main virtues of these axioms are their conceptual similarity to conventional large cardinals, combined with their ability to settle classical questions like the continuum hypothesis (CH).  Foreman noted in \cite{foremanhandbook} that some trouble was raised for this proposal by the discovery of a few pairs of these axioms that were mutually inconsistent.  For example, Woodin showed that under CH, one cannot have both an $\omega_1$-dense ideal on $\omega_1$ and a normal ideal on $[\omega_2]^{\omega_1}$ with quotient algebra isomorphic to $\col(\omega,< \! \omega_2)$.  An $\omega_1$-dense ideal on $\omega_1$ is known to be consistent with CH from large cardinals, but no consistency proof is known for Woodin's second ideal.  However, Foreman noted that Woodin's argument works equally well if we replace $\omega_2$ by an inaccessible $\lambda$, and the consistency of this kind of ideal on $[\lambda]^{\omega_1}$ can be established from a huge cardinal.  Foreman remarked that such examples of individually consistent but mutually inconsistent generic large cardinals seem to be relatively rare.  This paper brings to light a new class of examples through a strengthening of Woodin's theorem:

\begin{theorem}
\label{incon1}
Suppose $\kappa$ is a successor cardinal and there is a $\kappa$-complete, $\kappa$-dense ideal on $\kappa$.  Then neither of the following hold:
\begin{enumerate}[(1)]
\item There is a weakly inaccessible $\lambda > \kappa$ and a normal, $\kappa$-complete, $\lambda^+$-saturated ideal on $[\lambda]^\kappa$.
\item There is a successor cardinal $\lambda > \kappa$ and a normal, $\kappa$-complete, $\lambda$-saturated on $[\lambda]^\kappa$, where the forcing associated to the ideal has uniform density $\lambda$ and preserves $\lambda^+$-saturated ideals on $\lambda$.
\end{enumerate}
\end{theorem}

The ideal on $[\omega_2]^{\omega_1}$ that Woodin hypothesized will be seen to be a special case of (2), and we will not need cardinal arithmetic assumptions in our argument.  We will also give a relatively simple proof of the individual consistency of (1) and (2) from a huge cardinal, and even their joint consistency from a huge cardinal with two targets.  For the individual consistency of dense ideals on general successors of regular cardinals, we refer to \cite{thesis}.  

Let us fix some terminology and notation.  For any partial order $\mathbb{P}$, the saturation of $\mathbb{P}$, $\sat(\mathbb{P})$, is the least cardinal $\kappa$ such that every antichain in $\mathbb{P}$ has size less than $\kappa$.  The density of $\mathbb{P}$, $\den(\mathbb{P})$, is the least cardinality of a dense subset of $\mathbb{P}$.  Clearly $\sat(\mathbb{P}) \leq \den(\mathbb{P})^+$ for any $\mathbb{P}$.  We say $\mathbb{P}$ is $\kappa$-saturated if $\sat(\mathbb{P}) \leq \kappa$, and $\mathbb{P}$ is $\kappa$-dense if $\den(\mathbb{P}) \leq \kappa$.  A synonym for $\kappa$-saturation is the $\kappa$ chain condition ($\kappa$-c.c.).


For an ideal $I$ on a set $Z$, we let $I^+$ denote subsets of $Z$ not in $I$.  $\p(Z)/I$ denotes the boolean algebra of subsets of $Z$ modulo $I$ under the usual operations, where two sets are equivalent if their symmetric difference is in $I$.  The saturation and density of $I$ refer to that of $\p(Z)/I$.  For preliminaries on forcing with ideals, we refer the reader to~\cite{foremanhandbook}.  The following is the key backdrop for this paper:

\begin{proposition}Suppose $I$ is a normal, $\lambda^+$-saturated ideal on $Z \subseteq \p(\lambda)$.  If $G \subseteq \p(Z)/I$ is generic, the ultrapower $V^Z/G$ is isomorphic to a transitive class $M \subseteq V[G]$.  Suppose $I$ is $\kappa$-complete but $I \restriction A$ is not $\kappa^+$-complete for any $A \in I^+$.  If $j : V \to M$ is the canonical elementary embedding, then $\crit(j) = \kappa$, $[id]_G = j[\lambda]$, and $M^\lambda \cap V[G] \subseteq M$.
\end{proposition}

Recall that $[\lambda]^\kappa$ denotes $\{ z \subseteq \lambda : \ot(z) = \kappa \}$.  If $j : V \to M$ is an embedding arising from a normal $\lambda^+$-saturated ideal on $[\lambda]^\kappa$, then by \L o\'{s}'s theorem, $\lambda = \ot(j[\lambda]) = \ot([id]) = j(\kappa)$.

\section{Duality}

In this section we present a generalization of Foreman's Duality Theorem~\cite{foremanduality} that will be used for both the mutual inconsistency and individual consistency of the relevant principles.

\begin{theorem}
\label{dualitygen}
Suppose $I$ is a precipitous ideal on $Z$ and $\mathbb{P}$ is a boolean algebra.  Let $j: V \to M \subseteq V[G]$ denote a generic ultrapower embedding arising from $I$.  Suppose $\dot{K}$ is a $\p(Z)/I$-name for an ideal on $j(\mathbb P)$ such that whenever $G * h$ is $\p(Z)/I * j(\mathbb P)/ K$-generic and $\hat H = \{ p : [p]_K \in h \}$, we have:

\begin{enumerate}[(1)]
\item $1 \Vdash_{\p(Z)/I * j(\mathbb P)/ K} \hat{H}$ is $j(\mathbb{P})$-generic over $M$,
\item $1 \Vdash_{\p(Z)/I * j(\mathbb P)/ K} j^{-1}[\hat{H}]$ is $\mathbb{P}$-generic over $V$, and
\item for all $p \in \mathbb{P}$, $1 \nVdash_{\p(Z)/I} j(p) \in K$.
\end{enumerate}

Then there is $\mathbb{P}$-name for an ideal $J$ on $Z$ and a canonical isomorphism
\[ \iota : \mathcal{B}( \mathbb{P} * \dot{\p(Z)/J}) \cong \mathcal{B}( \p(Z)/I * \dot{j(\mathbb{P})/K} ). \]
\end{theorem}
\begin{proof}
Let $e : \mathbb{P} \to \mathcal{B}(\p(Z)/I * j(\mathbb{P})/K)$ be defined by $p \mapsto || j(p) \in \hat{H} ||$.  By (3), this map has trivial kernel.  By elementarity, it is an order and antichain preserving map.  If $A \subseteq \mathbb{P}$ is a maximal antichain, then it is forced that $j^{-1}[\hat{H}] \cap A \not= \emptyset$.  Thus $e$ is regular.

Whenever $H \subseteq \mathbb{P}$ is generic, there is a further forcing yielding a generic $G * h \subseteq \p(Z)/I * j(\mathbb{P})/K$ such that $j[H] \subseteq \hat{H}$.  Thus there is an embedding $\hat{j} : V[H] \to M[\hat{H}]$ extending $j$.  In $V[H]$, let $J = \{ A \subseteq Z : 1 \Vdash_{(\p(Z)/I * j(\mathbb{P})/K) / e[H]} [id]_M \notin \hat{j}(A) \}$.  In $V$, define a map $\iota : \mathbb{P} * \dot{\p(Z)/J} \to \mathcal{B}( \p(Z)/I * \dot{j(\mathbb{P})/K} )$ by $(p,\dot{A}) \mapsto e(p) \wedge || [id]_M \in \hat{j}(\dot{A}) ||$.  It is easy to check that $\iota$ is order and antichain preserving.

We want to show the range of $\iota$ is dense.  Let $(B,\dot{q}) \in \p(Z)/I * \dot{j(\mathbb{P})/K}$, and WLOG, we may assume there is some $f : Z \to V$ in $V$ such that $B \Vdash \dot{q} = [[f]_M]_K$.  By the regularity of $e$, let $p \in \mathbb{P}$ be such that for all $p' \leq p$, $e(p') \wedge (B,\dot{q}) \not= 0$.  Let $\dot{A}$ be a $\mathbb{P}$-name such that $p \Vdash \dot{A} = \{ z \in B : f(z) \in H \}$, and $\neg p \Vdash \dot{A} = Z$.  $1 \Vdash_\mathbb{P} \dot{A} \in J^+$ because for any $p' \leq p$, we can take a generic $G * h$ such that $e(p') \wedge (B, \dot{q}) \in G*h$.  Here we have $[id]_M \in j(B)$ and $[f]_M \in \hat{H}$, so $[id]_M \in \hat{j}(A)$.  Furthermore, $\iota(p,\dot{A})$ forces $B \in G$ and $q \in h$, showing $\iota$ is a dense embedding.
\end{proof}

\begin{proposition}
\label{ultequal}
If $Z,I,\mathbb{P},J,K,\iota$ are as in Theorem~\ref{dualitygen}, then whenever $H \subseteq \mathbb{P}$ is generic, $J$ is precipitous and has the same completeness and normality that $I$ has in $V$.  Also, if $\bar{G} \subseteq \p(Z)/ J$ is generic and $G * h = \iota[H * \bar{G}]$, then if $\hat{j} : V[H] \to M[\hat{H}]$ is as above, $M[\hat{H}] = V[H]^Z/\bar{G}$ and $\hat{j}$ is the canonical ultrapower embedding.
\end{proposition}
\begin{proof}
Suppose $H * \bar{G} \subseteq  \mathbb{P} * \p(Z) / J$ is generic, and let $G * h = \iota[H * \bar{G}]$ and $\hat{H} = \{ p : [p]_K \in h \}$.  For $A \in J^+$, $A \in \bar{G}$ iff $[id]_M \in \hat{j}(A)$.  If $i : V[H] \to N = V[H]^Z / \bar{G}$ is the canonical ultrapower embedding, then there is an elementary embedding $k : N \to M[\hat{H}]$ given by $k([f]_N) = \hat{j}(f)([id]_M)$, and $\hat{j} = k \circ i$.  Thus $N$ is well-founded, so $J$ is precipitous.  If $f : Z \to \ord$ is a function in $V$, then $k([f]_N) = j(f)([id]_M) = [f]_M$.  Thus $k$ is surjective on ordinals, so it must be the identity, and $N = M[\hat{H}]$.  Since $i = \hat{j}$ and $\hat{j}$ extends $j$, $i$ and $j$ have the same critical point, so the completeness of $J$ is the same as that of $I$.  Finally, since $[id]_N = [id]_M$, $I$ is normal in $V$ iff $J$ is normal in $V[H]$, because $j \restriction \bigcup Z = \hat{j} \restriction \bigcup Z$, and normality is equivalent to $[id] = j[\bigcup Z]$.
\end{proof}

Theorem~\ref{dualitygen} is optimal in the sense that it characterizes exactly when an elementary embedding coming from a precipitous ideal can have its domain enlarged via forcing:

\begin{proposition}
Let $I$ be a precipitous ideal on $Z$ and $\mathbb{P}$ a boolean algebra.  The following are equivalent.
\begin{enumerate}[(1)]
\item In some generic extension of a $\p(Z)/I$-generic extension, there is an elementary embedding $\hat j : V[H] \to M[\hat H]$, where $j : V \to M$ is the elementary embedding arising from $I$ and $H$ is  $\mathbb P$-generic over $V$.
\item There are  $p \in \mathbb P$, $A \in I^+$, and a $\p(A)/I$-name for an ideal $K$ on $j(\mathbb P \restriction p)$ such that $\p(A)/I * j(\mathbb P \restriction p)/K$ satisfies the hypothesis of Theorem \ref{dualitygen}.
\end{enumerate}
\end{proposition}
\begin{proof}
$(2) \Rightarrow (1)$ is trivial.  To show $(1) \Rightarrow (2)$, let $\dot{\mathbb Q}$ be a $\p(Z)/I$-name for a partial order, and suppose $A \in I^+$ and $\dot{H}_0$ are such that $\Vdash_{\p(A)/I * \dot{\mathbb Q}}$ ``$\dot{H}_0$ is $j(\mathbb P)$-generic over $M$ and $j^{-1}[\dot{H}_0]$ is $\mathbb P$-generic over $V$.''  By the genericity of $j^{-1}[\dot{H}_0]$, the set of $p \in \mathbb P$ such that $\Vdash_{\p(A)/I * \dot{\mathbb Q}} j(p) \notin \dot{H}_0$ is not dense.  So let $p_0$ be such that for all $p \leq p_0$, $|| j(p) \in \dot{H}_0 || \not = 0$.  In $V^{\p(A)/I}$, define an ideal $K$ on $j(\mathbb P \restriction p_0)$ by $K = \{ p \in j(\mathbb P \restriction p_0) : 1 \Vdash_\mathbb{Q} p \notin \dot{H}_0 \}$.  We claim $K$ satisfies the hypotheses of Theorem \ref{dualitygen}.  Let $G * h$ be $\p(A)/I * j(\mathbb P \restriction p_0)/K$-generic.  In $V[G*h]$, let $\hat{H} = \{ p \in j(\mathbb{P}) : [p]_K \in h \}$.

\begin{enumerate}[(1)]
\item If $D \in M$ is open and dense in $j(\mathbb{P})$, then $\{ [d]_K : d \in D$ and $d \notin K \}$ is dense in $j(\mathbb{P})/K$.  For otherwise, there is $p \in j(\mathbb{P}) \setminus K$ such that $p \wedge d \in K$ for all $d \in D$.  By the definition of $K$, we can force with $\mathbb{Q}$ over $V[G]$ to obtain a filter $H_0 \subseteq j(\mathbb{P})$ with $p \in H_0$.  But $H_0$ cannot contain any elements of $D$, so it is not generic over $M$, a contradiction.  Thus if $h \subseteq j(\mathbb{P})/K$ is generic over $V[G]$, then $\hat{H}$ is $j(\mathbb{P})$-generic over $M$.
\item If $A \in V$ is a maximal antichain in $\mathbb{P}$, then $\{ [j(a)]_K : a \in A$ and $j(a) \notin K \}$ is a maximal antichain in $j(\mathbb{P})/K$.  For otherwise, there is $p \in j(\mathbb{P}) \setminus K$ such that $p \wedge j(a) \in K$ for all $a \in A$.  We can force with $\mathbb{Q}$ over $V[G]$ to obtain a filter $H_0 \subseteq j(\mathbb{P})$ with $p \in H_0$.  But $H_0$ cannot contain any elements of $j[A]$, so $j^{-1}[H_0]$ is not generic over $V$, a contradiction.
\item If $p \in \mathbb{P} \restriction p_0$, $|| j(p) \in \dot{H}_0 ||_{\p(A)/I * \dot{\mathbb Q}} \not=0$, so $1 \nVdash_{\p(Z)/I} j(p) \in K$.
\end{enumerate}
\end{proof}

\begin{lemma}
\label{generated}
Suppose the ideal $K$ in Theorem~\ref{dualitygen} is forced to be principal.  Let $\dot{m}$ be such that $\Vdash_{\p(Z)/I} \dot{K} = \{ p \in j(\mathbb{P}) : p \leq \neg \dot{m} \}$.  Suppose $f$ and $A$ are such that $A \Vdash \dot{m} = [f]$, and $\dot{B}$ is a $\mathbb{P}$-name for $\{ z \in A : f(z) \in H \}$.  Let $\bar{I}$ be the ideal generated by $I$ in $V[H]$.  Then $\bar{I} \restriction B = J \restriction B$, where $J$ is given by Theorem~\ref{dualitygen}.
\end{lemma}

\begin{proof}
Clearly $J \supseteq \bar{I}$.  Suppose that $p_0 \Vdash$ ``$\dot{C} \subseteq \dot{B}$ and $\dot{C} \in \bar{I}^+$,'' and let $p_1 \leq p_0$ be arbitrary.  WLOG $\mathbb{P}$ is a complete boolean algebra.  For each $z \in Z$, let $b_z = || z \in \dot{C} ||$.  In $V$, define $C' = \{ z : p_1 \wedge b_z \wedge f(z) \not= 0 \}$.  $p_1 \Vdash \dot{C} \subseteq C'$, so $C' \in I^+$.  If $G \subseteq \p(Z)/I$ is generic with $C' \in G$, then $j(p_1) \wedge b_{[id]} \wedge \dot{m} \not= 0$.  Take $\hat{H} \subseteq j(\mathbb{P})$ generic over $V[G]$ with $j(p) \wedge b_{[id]} \wedge \dot{m} \in \hat{H}$.  Since $b_{[id]} \Vdash^M_{j(\mathbb{P})} [id] \in \hat{j}(C)$, $p_1 \nVdash \dot{C} \in J$ as $p_1 \in H = j^{-1}[\hat{H}]$.  Thus $p_0 \Vdash \dot C \in J^+$.
\end{proof}

\begin{corollary}
\label{dualitynicecase}
If $I$ is a $\kappa$-complete precipitous ideal on $Z$ and $\mathbb{P}$ is $\kappa$-c.c.,\ then there is a canonical isomorphism $\iota : \mathbb{P} * \p(Z)/ \bar{I} \cong \p(Z)/I * j(\mathbb{P})$.
\end{corollary}
\begin{proof}
If $G * \hat{H} \subseteq \p(Z)/I * j(\mathbb{P})$ is generic, then for any maximal antichain $A \subseteq \mathbb{P}$ in $V$, $j[A] = j(A)$, and $M \models j(A)$ is a maximal antichain in $j(\mathbb{P})$.  Thus $j^{-1}[\hat{H}]$ is $\mathbb{P}$-generic over $V$, and clearly for each $p \in \mathbb{P}$, we can take $\hat{H}$ with $j(p) \in \hat{H}$.  Taking a $\p(Z)/I$-name $\dot K$ for the trivial ideal on $j(\mathbb P)$, Theorem~\ref{dualitygen} implies that there is a $\mathbb P$-name $\dot{J}$ for an ideal on $Z$ and an isomorphism $\iota : \mathcal{B}(\mathbb{P} * \p(Z) / J) \to \mathcal{B}(\p(Z)/I * j(\mathbb{P}))$, and Proposition~\ref{generated} implies that $J = \bar{I}$.
\end{proof}

\section{Mutual inconsistency}

\begin{lemma}
\label{satequiv}
Suppose $I$ is a normal ideal on $Z \subseteq \p(X)$ and $|Z| = |X|$.  The following are equivalent:
\begin{enumerate}[(1)]
\item $I$ is $|X|^+$-saturated.
\item Every normal $J \supseteq I$ on $Z$ is equal to $I \restriction A$ for some $A \subseteq Z$.
\item If $[A]_I \Vdash \tau \in V^Z/G$, then there is some function $f : Z \to V$ in $V$ such that $[A]_I \Vdash \tau = [\check f]_G$.
\end{enumerate}
\end{lemma}

\begin{proof}
We only use $|Z| = |X|$ for $(3) \Rightarrow (1)$.  To show $(1) \Rightarrow (2)$, suppose $I$ is $|X|^+$-saturated.  Let $\{ A_x : x \in X \}$ be a maximal antichain in $J \cap I^+$.  Then $[\nabla A_x]$ is the largest element of $\p(Z)/I$ whose elements are in $J$.  Thus $J = I \restriction (Z \setminus \nabla A_x)$.  For $(2) \Rightarrow (1)$, suppose $I$ is not $|X|^+$-saturated, and let $\{ A_\alpha : \alpha < \delta \}$ be a maximal antichain where $\delta \geq |X|^+$.  Let $J$ be the ideal generated by $\bigcup \{ \Sigma_{\alpha \in Y} [A_\alpha] : Y \in \p_{|X|^+}(\delta) \}$.  Then $J$ is a normal, proper ideal extending $I$.  $J$ cannot be equal $I \restriction A$ for some $A \in I^+$ because if so, there is some $\alpha$ where $A \cap A_\alpha \in I^+$.  $A \cap A_\alpha \in J$ by construction, but every $I$-positive subset of $A$ is $(I \restriction A)$-positive.

For the implication $(1) \Rightarrow (3)$, see Propositions 2.12 and 2.23 of \cite{foremanhandbook}.  For $(3) \Rightarrow (1)$ under the assumption $|Z| = |X|$, use Remark 2.13.
\end{proof}

Suppose $\kappa = \mu^+$ and $I$ is a normal, $\kappa$-complete, $\lambda^+$-saturated ideal on $Z \subseteq \p_\kappa(\lambda)$.  If $j : V \to M \subseteq V[G]$ is a generic embedding arising from $I$, then by \L o\'{s}'s theorem, $[id] < j(\kappa)$, and thus $|\lambda| = \mu$ in $V[G]$.  Since $\lambda^+$ is preserved, $j(\kappa) = \lambda^+$.  Now suppose $\mathbb P$ is a $\kappa$-c.c.\ partial order.  By Corollary~\ref{dualitynicecase}, $\mathbb P$ preserves the $\lambda^+$-saturation of $I$ just in case $\p(Z)/I$ forces $j(\mathbb P)$ is $\lambda^+$-c.c.  The notion of a \emph{c.c.c.-indestructible} ideal on $\omega_1$ has been considered before, for example in \cite{bt2} and \cite{fms1}, and a straightforward generalization would be to say that for successor cardinals $\kappa$, a normal $\kappa$-complete $\lambda^+$-saturated ideal on $Z \subseteq \p_\kappa(\lambda)$ is indestructible if its saturation is preserved by every $\kappa$-c.c.\ forcing.  For the next lemma however, we consider the dual notion by fixing a $\kappa$-c.c.\ partial order and quantifying over saturated ideals on a set $Z$.

\begin{definition}
$\mathbb P$ is $Z$-absolutely $\kappa$-c.c.\ when for all normal, $|\bigcup Z|^+$-saturated ideals $I$ on $Z$, $1 \Vdash_{\p(Z)/I} j(\mathbb P)$ is $j(\kappa)$-c.c.
\end{definition} 

Such partial orders abound; for example $\add(\omega,\alpha)$ is $Z$-absolutely c.c.c.\ for all $Z$ and $\alpha$, since Cohen forcing has an absolute definition.  Also, if $\kappa = \mu^+$ and $Z$ is such that every normal ideal on $Z$ is $\kappa$-complete, then all $\mu$-centered or $\mu$-c.c.\ partial orders are $Z$-absolutely $\kappa$-c.c.

Proof of the following basic lemma is left to the reader:
\begin{lemma}Suppose $e : \mathbb{P} \to \mathbb{Q}$ is a regular embedding and $\kappa$ is a cardinal.  $\mathbb{Q}$ is $\kappa$-dense iff $\mathbb{P}$ is $\kappa$-dense and $\Vdash_{\mathbb{P}} \mathbb{Q} / \dot{G}$ is $\kappa$-dense.
\end{lemma}

\begin{lemma}
\label{nodense}
Suppose $\kappa = \mu^+$ and $Z \subseteq \{ z \in \p_\kappa(\lambda) : z \cap \kappa \in \kappa \}$.  Suppose $\mathbb{P}$ is $Z$-absolutely $\kappa$-c.c.\ and $\kappa \leq \den(\mathbb{P} \restriction p) \leq \lambda$ for all $p \in \mathbb{P}$.  Then in $V^\mathbb{P}$, there are no normal $\lambda$-dense ideals on $Z$.
\end{lemma}

\begin{proof}Suppose $p \Vdash \dot{J}$ is a normal, fine, $\kappa$-complete, $\lambda^+$-saturated ideal on $Z$.  Let $I = \{ X \subseteq Z : p \Vdash X \in \dot{J} \}$.  It is easy to check that $I$ is normal.  The map $\sigma : \p(Z) / I \to \mathcal{B}(\mathbb{P} \restriction p * \p(Z) / \dot{J})$ that sends $X$ to $( || \check{X} \in \dot{J}^+ || , \dot{[X]_J} )$ is an order-preserving and antichain-preserving map, so $I$ is $\lambda^+$-saturated.

Let $H$ be $\mathbb{P}$-generic over $V$ with $p \in H$.  By Corollary~\ref{dualitynicecase}, $\p^{V[H]} (Z) / \bar{I} \cong (\p^V (Z) / I * j(\mathbb{P})) / e[H]$, where $e$ is the regular embedding from Theorem~\ref{dualitygen}, so $\bar I$ is $\lambda^+ = j(\kappa)$-saturated by the $Z$-absolute $\kappa$-c.c.   By Proposition~\ref{ultequal}, $\bar{I}$ is normal, and clearly $\bar{I} \subseteq J$.  By Lemma~\ref{satequiv}, there is $A \in \bar{I}^+$ such that $J = \bar{I} \restriction A$.  Since $j(\mathbb{P})$ is forced to be nowhere $<j(\kappa)$-dense, $\p (Z) / I * j(\mathbb{P})$ is nowhere $\lambda$-dense.  Since $\den(\mathbb{P}) \leq \lambda$, $(\p^V (Z) / I * j(\mathbb{P})) / e[H]$ is nowhere $\lambda$-dense.  Thus $J$ is not $\lambda$-dense. 
\end{proof}

We can now prove the mutual inconsistency of the following hypotheses when $\mu^+ = \kappa < \lambda = \eta^+$:

\begin{enumerate}[(1)]
\item There is a $\kappa$-complete, $\kappa$-dense ideal $I$ on $\kappa$.
\item There is a normal, $\kappa$-complete, $\lambda$-absolutely $\lambda$-saturated $J$ ideal on $[\lambda]^\kappa$, such that $\den(\p(A)/J) = \lambda$ for all $A \in J^+$.
\end{enumerate}

Solovay observed that if there is a $\kappa$-complete, $\kappa^+$-saturated ideal on $\kappa$, then there is a normal one as well.  This also applies to dense ideals.  For suppose $I$ is a $\kappa$-complete, $\kappa$-dense ideal on some set $Z$.  Let $J = \{ X \subseteq \kappa : 1 \Vdash_{\p(Z)/I} \kappa \notin j(X) \}$.  It is easy to verify that $J$ is normal and that the map $X \mapsto || \kappa \in j(X) ||$ is an order and antichain preserving map of $\p(\kappa)/J$ into $\p(Z)/I$.  If $\{ X_\alpha : \alpha < \kappa \}$ is a maximal antichain in $\p(\kappa)/I$, then $[\nabla X_\alpha]_J = [\kappa]_J$, so $1 \Vdash_{\p(Z)/I} (\exists \alpha < \kappa) \kappa \in j(X_\alpha)$.  Thus $\{ e(X_\alpha) : \alpha < \kappa \}$ is maximal in $\p(Z)/I$ and $J$ is $\kappa$-dense.

Suppose (1) and (2) hold.  If $j : V \to M \subseteq V[G]$ is a generic embedding arising from $\p([\lambda]^\kappa) / J$, then $j(\kappa) = \lambda$ and $M^\lambda \cap V[G] \subseteq M$.  $M$ thinks $j(I)$ is a normal, $\lambda$-dense ideal on $\lambda$, and this holds in $V[G]$ as well by the closure of $M$.  But since the forcing to produce $G$ is $\lambda$-absolutely $\lambda$-c.c.\ and of uniform density $\lambda$, Lemma~\ref{nodense} implies that no such ideals can exist in $V[G]$, a contradiction.

Under GCH, this part of the result can be strengthened slightly.  Suppose $\lambda^{<\lambda} = \lambda$ and $\mathbb P$ is $\lambda$-c.c.  If $\theta$ is sufficiently large, then we can take an elementary $M \prec H_\theta$ with $\mathbb P \in M$, $|M| = \lambda$, $M^{<\lambda} \subseteq M$.  Then any maximal antichain in $\mathbb P \cap M$ is a member of $M$, and by elementarity is also a maximal antichain in $\mathbb P$.  It is clear from the definition that in general, if $\mathbb P$ is $Z$-absolutely $\kappa$-c.c., then so is any regular suborder.  So if $\mathbb P$ is also nowhere $<\lambda$-dense and $\lambda$-absolutely $\lambda$-c.c., forcing with $\mathbb P \cap M$ will destroy all dense ideals on $\lambda$.

To see that the same holds for $\mathbb P$, suppose $\dot J$ is a $\mathbb P$-name for a normal ideal on $\lambda$ and $D = \{ \dot A_\alpha : \alpha < \lambda \}$ is a sequence of names for subsets of $\lambda$ witnessing that $\dot J$ is $\lambda$-dense.  Let $I = \{ X \subseteq \lambda : 1 \Vdash_{\mathbb P} X \in \dot J \}$.  As in the proof of Lemma~\ref{nodense}, $I$ is a normal $\lambda^+$-saturated ideal on $\lambda$, and it is forced that $\bar I = \dot J \restriction \dot A$ for some $\dot A \in \dot J^+$.  Let $\theta$ be sufficiently large, and take $M \prec H_\theta$ with $|M| = \lambda$, $\lambda \subseteq M$, $M^{<\lambda} \subseteq M$, $\{ \mathbb P, \dot J, \dot B, D \} \in M$.  We may assume that each $\dot A_\alpha$ and $\dot B$ take the form $\cup_{\beta < \lambda} S_\beta \times \{ \check \beta \}$, where each $S_\beta$ is a maximal antichain, so that they are all $\mathbb P \cap M$-names.

Let $G \subseteq \mathbb P$ be generic and let $G_0 = G \cap M$.  Let $I_0$ be the ideal generated by $I$ in $V[G_0]$, and let $I_1$ be the ideal generated by $I$ in $V[G]$.  In $V[G_0]$, if $C$ is an $I_0$-positive subseteq of $\dot B^{G_0}$, then in $V[G]$, there is some $A_\alpha$ such that $A_\alpha \setminus C \in J$.  But this just means that some $Y \in I$ covers $A_\alpha \setminus C$, and this is absolute to $V[G_0]$.  Thus $I_0 \restriction B$ is $\lambda$-dense in $V[G_0]$, contradicting Lemma~\ref{nodense}.  Thus we have:

\begin{proposition}
If $\mu^+ = \kappa < \lambda = \eta^+$ and $2^\eta = \lambda$, then there cannot exist both a $\kappa$-complete, $\kappa$-dense ideal on $\kappa$ and a normal, $\kappa$-complete, $\lambda$-absolutely $\lambda$-saturated, nowhere $\eta$-dense ideal on $[\lambda]^\kappa$.
\end{proposition}

Now we turn to the weakly inaccessible case.  Suppose that $\kappa = \mu^+$, there is a $\kappa$-dense normal ideal $I$ on $\kappa$, $\lambda > \kappa$ is weakly inaccessible, and there is a normal, $\kappa$-complete, $\lambda^+$-saturated ideal $J$ on $[\lambda]^\kappa$.  If $j : V \to M \subseteq V[G]$ is an embedding arising from $J$, then $M \models j(I)$ is a normal, $\lambda$-dense ideal on $\lambda$, and $V[G]$ satisfies the same by the closure of $M$.  In $V$, we may define a normal ideal $I^\prime = \{ X \subseteq \lambda : 1 \Vdash_{\p([\lambda]^\kappa)/J} X \in j(I) \}$.  As before, the map $X \mapsto ( || \check{X} \in j(I)^+ || , \dot{[X]_{j(I)}} )$ preserves antichains, implying $I^\prime$ is $\lambda^+$-saturated.

Now since $\lambda$ is weakly inaccessible in $V$, whenever $i : V \to M$ is an embedding arising from $I^\prime$, $i(\lambda) > \lambda^+$ since $M$ has the same $\lambda^+$ and thinks $i(\lambda)$ is a limit cardinal.  By Lemma~\ref{satequiv}, for each $\alpha \leq \lambda^+$, there is a function $f_\alpha : \lambda \to \lambda$ such that $1 \Vdash_{\p(\lambda)/I^\prime} \check \alpha = [f_\alpha]$.  This means that for $\alpha < \beta \leq \lambda^+$, $\{ \gamma : f_\alpha(\gamma) \geq f_\beta(\gamma) \} \in I^\prime$.

In $V[G]$, $I^\prime \subseteq j(I)$.  If $H \subseteq \p(\lambda)/j(I)$ is generic over $V[G]$, and $k : V[G] \to N$ is the associated embedding, then $N \models [f_\alpha]_H < [f_\beta]_H < k(\lambda)$ for $\alpha < \beta \leq \lambda^+$.  Thus the ordertype of $k(\lambda)$ is greater than $\lambda^+$.  But since $j(I)$ is $\lambda^+$-saturated and $\lambda = \mu^+$ in $V[G]$, $k(\lambda) = \lambda^+$, a contradiction.

The $\kappa$-density of the ideal on $\kappa$ was only relevant to have a saturation property of an ideal on $\lambda$ that is upwards-absolute to a model with the same $\p(\lambda)$. Furthermore, if we replace the ideal $J$ on $[\lambda]^\kappa$ by one on $\{ z \subseteq \lambda^+ : \ot(z \cap \lambda) = \kappa \}$, then we get enough closure of the generic ultrapower to derive a contradiction from simply a $\kappa^+$-saturated ideal on $\kappa$.  (This kind of ideal on $\p(\lambda^+)$ will also be shown individually consistent from a large cardinal assumption in the next section.)  The common thread is captured by the following:

\begin{proposition}
\label{inaccincon}
Suppose $\kappa$ is a successor cardinal, $\lambda > \kappa$ is a limit cardinal, and $Z$ is such that $\lambda \subseteq \bigcup Z$, and for all $z \in Z$, $z \cap \kappa \in \kappa$ and $\ot(z \cap \lambda) = \kappa$.  Then the following are mutually inconsistent:
\begin{enumerate}[(1)]
\item There is a $\kappa$-complete, $Z$-absolutely $\kappa^+$-saturated ideal on $\kappa$.
\item There is a normal $\lambda^+$-saturated ideal on $Z$.
\end{enumerate}
\end{proposition}


\section{Consistency}

First we introduce a simple partial order.  The basic idea is due to Shioya~\cite{shioyanew}.  If $\mu < \kappa$ are in $\reg$ (the class of regular cardinals), we define:
\[ \mathbb P(\mu,\kappa) = \prod^{<\mu \supp}_{\alpha \in (\mu,\kappa) \cap \reg} \col(\alpha,< \! \kappa) \]
The ordering is reverse inclusion.  It is a bit more convenient to view $\mathbb P(\mu,\kappa)$ as the collection of partial functions with domain contained in $\kappa^3$ such that:
\begin{itemize}
\item $\forall (\alpha,\beta,\gamma) \in \dom p$, $\alpha$ is regular, $\mu < \alpha < \kappa$, $\gamma < \alpha$, and $p(\alpha,\beta,\gamma)< \beta$.
\item $| \{ \alpha : \exists \beta \exists \gamma (\alpha,\beta,\gamma) \in \dom p \} | < \mu$.
\item $\forall \alpha | \{ (\beta,\gamma) : (\alpha,\beta,\gamma) \in \dom p \} | < \alpha$.
\end{itemize}
Proof of the next two lemmas is standard.

\begin{lemma}
If $\mu$ is regular and $\kappa > \mu$ is inaccessible, then $\mathbb P(\mu,\kappa)$ is $\kappa$-c.c., $\mu$-closed, and of size $\kappa$.
\end{lemma}

\begin{lemma}If $\mu \leq \kappa \leq \lambda$, then there is a projection  $\sigma : \mathbb P(\mu,\lambda) \to \mathbb P(\mu,\kappa)$ given by $\sigma(p) = p \restriction \kappa^3$.
\end{lemma} 

Suppose that there is elementary embedding $j : V \to M$ with $\crit(j) = \kappa$ and $M^\lambda \subseteq M$.  If $\lambda = j(\kappa)$, then $\kappa$ is called \emph{huge}.  A slightly stronger hypothesis is that $\lambda = j(\kappa)^+$.  This is substantially weaker than a \emph{2-huge} cardinal, which is when $\lambda = j^2(\kappa)$.  The hugeness of $\kappa$ with target $\lambda$ is equivalent to the existence of a normal $\kappa$-complete ultrafilter on $[\lambda]^\kappa$, and a shorthand for this is ``$\kappa$ is $\lambda$-huge.''  The slightly stronger hypothesis is equivalent to the existence of a normal $\kappa$-complete ultrafilter on $\{ z \subseteq \lambda^+ : \ot(z \cap \lambda) = \kappa \}$.  The next result shows the relative consistency of the type of ideals mentioned in (1) of Theorem~\ref{incon1} and in the discussion preceding Proposition~\ref{inaccincon}.

\begin{proposition}
\label{inacccon}
Suppose that $\lambda> \kappa$ and there is a $\kappa$-complete normal ultrafilter on $Z$ such that $\forall z \in Z (\ot(z \cap \lambda) = \kappa)$. If $\mu< \kappa$ is regular and $G \subseteq \mathbb P(\mu,\kappa)$ is generic, then in $V[G]$, $\kappa = \mu^+$, $\lambda$ is inaccessible, and there is a normal, $\kappa$-complete, $\lambda$-saturated ideal $J$ on $Z$, such that $\p(Z)/J \cong \mathbb P(\mu,\lambda)/G$.
\end{proposition}

\begin{proof}Let $I$ be the dual ideal to a $\kappa$-complete normal ultrafilter on $Z$ and $j$ the associated embedding.  By Corollary~\ref{dualitynicecase}, $\mathbb P(\mu,\kappa) * \p(Z)/\bar I \cong \mathbb P(\mu,\lambda)$.  Since $j \circ \sigma = id$, $\mathbb P(\mu,\lambda) / \sigma^{-1}[G] = \mathbb P(\mu,\lambda) / j[G]$, so the map given by Theorem~\ref{dualitygen} shows the desired isomorphism.   The normality and completeness claims follow from Proposition~\ref{ultequal}.  It is clear that $\mathbb P(\mu,\lambda)$ forces $\kappa = \mu^+$ and preserves the inaccessbility of $\lambda$.  
\end{proof}

\begin{lemma}[Shioya~\cite{shioyanew}]
Assume GCH, $\kappa$ is regular, $\lambda \leq \kappa$ is inaccessible, and $\mathbb P$ is $\kappa$-c.c.\ and of size $\leq \kappa$.  If $G \subseteq \mathbb P$ is generic, then in $V[G]$ there is a projection $\pi : \col(\kappa,< \! \lambda)^V \to \col(\kappa,< \! \lambda)^{V[G]}$.
\end{lemma}

\begin{proof}
First note that $\col(\kappa,\lambda)$ is isomorphic to the subset of conditions $q$ such that $\dom(q) \subseteq \reg \times \kappa$, since $\col(\kappa,\eta^+) \cong \prod^{<\kappa \supp}_{\alpha \in [\eta,\eta^+]} \col(\kappa,\alpha)$ for all $\eta$.  So we work with this partial order instead.

Inductively choose a sequence of $\mathbb P$-names for ordinals below $\lambda$, $\langle \tau_\alpha : \alpha < \lambda \rangle$, such that for every regular $\eta \in [\kappa,\lambda]$ and every $\mathbb P$-name $\sigma$, if $\Vdash \sigma < \check \eta$, then there is $\alpha < \eta$ such that $\Vdash \sigma = \tau_\alpha$.  For a given $q \in \col(\kappa,\lambda)^V$, let $\tau_q$ be the $\mathbb P$-name for a function such that $\Vdash \dom(\tau_q) = \dom(\check q)$, and for all $(\alpha,\beta) \in \dom(q)$, $\Vdash \tau_q(\alpha,\beta) = \tau_{q(\alpha,\beta)}$.  In $V[G]$, let $\pi(q) = \tau_q^G$.

To show $\pi$ is a projection, suppose $p \in G$ and $p \Vdash \dot q_1 \leq \pi(q_0)$.  By the $\kappa$-c.c., there is some $d \subseteq \lambda \times \kappa$ such that $|d| < \kappa$ and $\Vdash \dom \dot q_1 \subseteq \check d$.  Working in $V$, if $(\alpha,\beta) \in d \setminus \dom( q_0)$, then there is a name $r(\alpha,\beta)$ for an ordinal $< \alpha$ such that $p \Vdash (\alpha,\beta) \in \dom(\dot q_1) \rightarrow \dot q_1(\alpha,\beta) = r(\alpha,\beta)$.  Then $r \cup q_0 = q_2 \leq q_0$, and $p \Vdash \tau_{q_2} \leq \dot q_1$.
\end{proof}

Now we show the relative consistency of the type of ideal mentioned in (2) of Theorem~\ref{incon1}:

\begin{theorem}
\label{succcon}
Assume GCH, $\mu < \kappa < \delta < \lambda$ are regular and $\kappa$ is $\lambda$-huge.  $\mathbb P(\mu,\kappa) * \dot{\col(\delta,< \! \lambda)}$ forces that $\kappa = \mu^+$, $\lambda = \delta^+$, and there is a normal, $\kappa$-complete, $\lambda$-absolutely $\lambda$-saturated ideal $J$ on $[\lambda]^\kappa$ of uniform density $\lambda$.

\end{theorem}

\begin{proof}
Let $I$ be the dual ideal to a $\kappa$-complete normal ultrafilter on $[\lambda]^\kappa$, and $j : V \to M$ the associated embedding.  It is easy to show that under $2^\lambda = \lambda^+$, $\lambda^+ < j(\lambda) < \lambda^{++}$.   Let $G \subseteq \mathbb P(\mu,\kappa)$ be generic over $V$.  By Proposition~\ref{inacccon}, in $V[G]$, $\bar I$ is a $\kappa$-complete normal ideal on $[\lambda]^\kappa$ such that $\p([\lambda]^\kappa)/ \bar I \cong \mathbb P(\mu,\lambda)/ \sigma^{-1}[G]$.

If $\hat G \subseteq P(\mu,\lambda)$ is generic extending $G$, then we can extend the embedding to $\hat j : V[G] \to M[\hat G]$, and by Proposition~\ref{ultequal}, $\hat j$ is also a generic ultrapower embedding arising from $\bar I$.  There is a projection from $P(\mu,\lambda)/ G$ to $\col(\delta, < \! \lambda)^V$, and by Shioya's lemma, there is in $V[G]$ a projection from $\col(\delta, < \! \lambda)^V$ to $\col(\delta, < \! \lambda)^{V[G]}$.  Let $H \subseteq \col(\delta,\lambda)^{V[G]}$ be the generic thus projected from $\hat G$.  Since $M[\hat G]^\lambda \cap V[\hat G] \subseteq M[\hat G]$, and $\col(\delta,< \! \lambda)$ is $\delta$-directed closed and of size $\lambda$, we may take in $M[G]$ a condition $m \in \col(j(\delta), < \! j(\lambda))^{M[\hat G]}$ that is stronger than $j(q)$ for all $q \in H$.  Since $j(\lambda)$ is inaccessible in $M[\hat G]$, and $j(\lambda)<\lambda^{++}$, we may list all the dense open subsets of $\col(j(\delta), < \! j(\lambda))^{M[\hat G]}$ that live in $M[\hat G]$ in ordertype $\lambda^+$ and build a filter $\hat H$ generic over $M[\hat G]$ with $m \in \hat H$.

Thus we have a $\p([\lambda]^\kappa)/\bar I$-name for a filter $\hat H \subseteq \col(j(\delta), < \! j(\lambda))$ that is generic over $M[\hat G]$, and such that $\hat j^{-1}[\hat H]$ is $\col(\delta,< \! \lambda)$-generic over $V[\hat G]$.  Theorem~\ref{dualitygen} implies that in $V[G][H]$, there is a normal, $\kappa$-complete ideal $J$ on $[\lambda]^\kappa$ with quotient algebra isomorphic to $\mathbb P(\mu,\lambda)/(G * H)$.  Because coordinates in $[\kappa,\lambda) \setminus \{ \delta \}$ are ignored by the projection, this forcing has uniform density $\lambda$.

It remains to show the absoluteness of the $\lambda$-c.c.  First we note the following properties of $\mathbb P(\mu,\lambda)$ in $V$:
\begin{enumerate}[(a)]
\item For all $\eta < \lambda$, $| \{ p \restriction \eta^3 : p \in \mathbb P(\mu,\lambda) \} | < \lambda$.
\item If $\eta < \lambda$ is regular and $\dom p \subseteq \eta \times \lambda^2$, then the ordertype of $\dom p$ in the Gödel ordering on triples is less than $\eta$.  This is simply because $|p| < \eta$.
\end{enumerate}
Now work in $V^\prime = V[G][H]$.   Suppose $K$ is a normal $\lambda^+$-saturated ideal on $\lambda$, let $A \subseteq \p(\lambda)/K$ be generic, let $i : V^\prime \to N$ be the associated embedding.  It suffices to show that $i(\mathbb P(\mu,\lambda)^V)$ is $\lambda^+$-c.c.\ in $V^\prime[A]$.  This is because the nature of the projection from $\mathbb P(\mu,\lambda)$ to $\mathbb P(\mu,\kappa) * \dot{\col(\delta,< \! \lambda)}$ gives that two functions in $\mathbb P(\mu,\lambda) / (G*H)$ are incompatible just when they disagree at some point in their common domain, which is the same criterion for incompatibility in $\mathbb P(\mu,\lambda)$.  We will make a $\Delta$-system argument, but using the more absolute properties (a) and (b) above instead of pure cardinality considerations.

Let $\{ p_\xi : \xi < \lambda^+ \}$ be a set of conditions in $i(\mathbb P(\mu,\lambda)^V)$.  Since $\mu < \lambda^+$ and we have GCH, we can take some $X \in [\lambda^+]^{\lambda^+}$ such that $\{ \{ \alpha : \exists \beta \exists \gamma (\alpha,\beta,\gamma) \in p_\xi \} : \xi \in X \}$ forms a $\Delta$-system with root $r$.   Denote $\dom(p_\xi \restriction r \times (\lambda^+)^2)$ by $d_\xi$.  There are unboundedly many $\eta < \lambda$ satisfying (b), so let $\sup r < \eta < \lambda^+$ be such that for all $\xi \in X$, the ordertype of $d_\xi$ in the Gödel ordering on $(\lambda^+)^3$ is less than $\eta$.  WLOG we may assume $\ot(d_\xi) = \eta$ for all $\xi \in X$.  Let $\langle d_\xi(\nu) : \nu < \eta \rangle$ list the elements of $d_\xi$ in increasing order.  If there is no $\nu$ such that $\{ d_\xi(\nu) : \xi \in X \}$ is unbounded in the Gödel ordering, then there is some $\zeta < \lambda^+$ such that $d_\xi \subseteq \zeta^3$ for all $\xi \in X$.  But by property (a) and elementarity, there are $<\lambda^+$ many distinct elements of $i(\mathbb P(\mu,\lambda)^V) \restriction \eta^3$, so we do not have an antichain.  Otherwise, let $\nu_0$ be the least ordinal $<\eta$ such that $\{ d_\xi(\nu_0) : \xi \in X \}$ is unbounded.  We can recursively choose $\langle \xi_\alpha : \alpha < \lambda^+ \rangle$ such that for $\alpha < \beta$, $d_{\xi_\alpha}(\nu_0)$ is above all elements of $d_{\xi_\beta}(\nu_0)$ in the Gödel ordering.  Let $\zeta$ be such that $\{ d_{\xi_\alpha}(\nu) : \alpha < \lambda^+$ and $\nu < \nu_0 \} \subseteq \zeta^3$.  Again by property (a), there must be $\alpha < \beta$ such that $p_{\xi_\alpha}$ is compatible with $p_{\xi_\beta}$.  This proves the $\lambda$-absolute $\lambda$-c.c.
\end{proof}


Using a technique of Magidor described in \cite{foremanhandbook}, one can show that if $\kappa$ is $\lambda$-huge (or just $\lambda$-almost huge) and $\mu< \kappa$ is regular, then there is a $\kappa^+$-saturated ideal on $\kappa$ after forcing with $\mathbb P(\mu,\kappa) * \dot{\col(\kappa,<\!\lambda)}$.  It is unknown whether there can exist a successor cardinal $\kappa$ and a normal, $\kappa$-complete, $\kappa^+$-saturated ideal on $[\kappa^+]^\kappa$.  In the above construction,  the closest we can come is a $\kappa^{++}$-saturated ideal on $[\kappa^{++}]^\kappa$.  The essential reason for this is not in getting the master condition $m$, as this problem can be overcome by constructing $\mathbb P(\mu,\kappa)$ with the Silver collapse instead of the Levy collapse.  Rather, the construction of the filter $\hat H$ seems blocked when we collapse $\lambda$ to  $\kappa^+$.


Finally, we note the mutual consistency of these different ``generic hugeness'' properties of a successor cardinal:
\begin{proposition}
Assume GCH, $\mu < \kappa < \delta < \lambda_0 < \lambda_1$ are regular, and $\kappa$ is both $\lambda_0$-huge and $\lambda_1$-huge.  Then there is a generic extension in which $\kappa = \mu^+$, $\lambda_0 = \delta^+$, $\lambda_1$ is inaccessible, and:
\begin{enumerate}[(1)]
\item There is a normal, $\kappa$-complete, $\lambda_0$-saturated ideal on $[\lambda_0]^\kappa$.
\item There is a normal, $\kappa$-complete, $\lambda_1$-saturated ideal on $[\lambda_1]^\kappa$.
\end{enumerate}
\end{proposition}
\begin{proof}
Let $G * H$ be $\mathbb P(\mu,\kappa) * \dot{\col(\delta, < \! \lambda_0)}$-generic.  Then (1) holds by Theorem~\ref{succcon}.  To show (2), we argue almost exactly the same as in the proof of Theorem~\ref{succcon}, noting that $\col(\delta,<\!\lambda_1)^V$ projects to $\col(\delta,<\!\lambda_0)^V$. 
\end{proof}

\bibliographystyle{amsplain.bst}
\bibliography{masterbib}

\end{document}